\documentclass[12pt]{article}
\RequirePackage[colorlinks,citecolor=blue,urlcolor=blue]{hyperref}

\usepackage{latexsym,amssymb,amsmath,amsfonts,amsthm}
\usepackage{amsthm,amsmath}
\usepackage{float}
\usepackage{enumerate}
\usepackage{epic}
\usepackage{bbm}
\usepackage{subfigure}
\usepackage{mathbbol}
\usepackage{graphicx}
\usepackage[toc,page]{appendix}
\usepackage{multirow}
\usepackage{amsfonts}
\usepackage[all]{xy}
\usepackage{caption}

\newtheorem{theorem}{Theorem}

\newtheorem{lemma}{Lemma}[section]

\newtheorem{prop}{Proposition}
\newtheorem{corollary}{Corollary}

\def\beq{ \begin{equation} }
\def\eeq{ \end{equation} }

\def\ep{\epsilon}

\def\square{\vcenter{\vbox{\hrule height .4pt
  \hbox{\vrule width .4pt height 5pt \kern 5pt
        \vrule width .4pt} \hrule height .4pt}}}

\def\RR{\mathbb{R}}
\def\ZZ{\mathbb{Z}}

\begin{document}
\title{Weak Convergence of a Seasonally Forced Stochastic Epidemic Model}
\author{Alun Lloyd and Yuan Zhang}
\maketitle

\begin{abstract}
In this study we extend the results of Kurtz (1970,1971) to show the weak convergence of epidemic processes that include explicit time dependence, specifically where the transmission parameter, $\beta(t)$, carries a time dependency. We first show that when population size goes to infinity, the time inhomogeneous process converges weakly to the solution of the mean-field ODE. Our second result is that, under proper scaling, the central limit type fluctuations converge to a diffusion process.
\end{abstract}

\section{Introduction}
Much of mathematical epidemiology draws upon deterministic descriptions of disease transmission processes \cite{Anderson91}. Deterministic models are attractive, in part, because they are easy to analyze and simulate. The importance of stochastic effects on disease transmission processes has, however, long been appreciated \cite{Bailey50,Bartlett56,Kendall56} and so it is natural to ask about the relationship between stochastic and deterministic models of a given process. 

Kurtz \cite{KZ} showed, for a general class of population models, weak convergence of the stochastic model to the corresponding deterministic model as the system size, $N$, tends to infinity. Further, in \cite{KZ71} Kurtz provided a central limit theorem-type result that explored the nature of this convergence in more detail, revealing a diffusion process behavior.   These limiting results have been used to justify the use of the multivariate normal approximation introduced by \cite{Whittle57} to close moment equations for nonlinear population models. Use of this approach (e.g. \cite{Isham91,Lloyd04}) allows one to assess the magnitude of stochastic fluctuations likely to be seen about the deterministic solution and hence determine the adequacy of a deterministic description.

Many epidemic processes, however, have an explicit time dependence \cite{Grassly06}. For instance, an infectious agent may be more transmissible at certain times of the year than at others. Several biological and social mechanisms can give rise to such seasonality, including sensitivity of certain viruses to humidity and congregation of children during school sessions. 

In this study, we extend the results of Kurtz to epidemic processes that include explicit time dependence, specifically where the transmission parameter, $\beta(t)$, carries a time dependency. The paper is organized as follows: in Section 2 we introduce the model. Section 3 provides a weak convergence result and in Section 4 a central limit theorem-type result is given. Simulation results are shown in Section 5. 

\section{The Model}

We consider the seasonal SIR (susceptible/infective/recovered) model taking value $(S^N(t),I^N(t), R^N(t))\in (\ZZ^+\cup \{0\})^3$ with transition rates as follows: 
\begin{table}[!hbp]
\center
\begin{tabular}{|c|c|c|c|}
\hline 
Event & Transition & Rate at which event occurs \\
\hline
Birth & $S\rightarrow S+1$ & $\nu (S+I+R)$\\
\hline
Susceptible Death  & $S\rightarrow S-1$ & $\nu S$\\
\hline
Infection & $S\rightarrow S-1, \  I\rightarrow I+1$ & $\beta(t)SI/(S+I+R)$\\
\hline
Recovery & $I\rightarrow I-1, \ R\to R+1$ & $\gamma I$\\
\hline
Infectious Death & $I\rightarrow I-1$ & $\nu I$\\
\hline
Recovered Death & $R\rightarrow R-1$ & $\nu R$\\
\hline
\end{tabular}
\caption{Transition rates for the seasonal SIR model}
\end{table}

\noindent Here, $\nu$ denotes the {\it per capita} birth and death rate, which we assume to be equal. $\gamma$ denotes the {\it per capita} recovery rate, implying that the average duration of infection is $1/\gamma$. $N$ is equal to the initial population size, $N=S(0)+I(0)+R(0)$, although, as discussed further below, it should be noted that the population size is not constant for this model, despite having equal {\it per capita} demographic parameters.

The transmission parameter, $\beta(t)$ is assumed to be a periodic function with period one. For definiteness, we take the following sinusoidal form: $\beta(t)=\beta_0[1+\beta_1\cos(2\pi t)]$, where $\beta_0>0$ and $\beta_1\in(0,1)$. It should be noted, however, that our results apply to a much broader class of functions.

From the transition rates above, it is easy to see that the stochastic model above is associated with the following mean-field ODE:
\beq
\label{MnODE}
\begin{aligned}
&\frac{d x_t}{d t}=F_1(x_t,y_t,z_t,t)= \nu(y_t+z_t)-\beta(t)x_ty_t\\
&\frac{d y_t}{d t}=F_2(x_t,y_t,z_t,t)=\beta(t)x_ty_t-(\nu+\gamma)y_t\\
&\frac{d z_t}{d t}=F_3(x_t,y_t,z_t,t)=\gamma y_t- \nu z_t. 
\end{aligned}
\eeq
where $x_t$, $y_t$ and $z_t$ now represent the fractions of the population in each state.

\section{Weak Convergence to the Solution of ODE}
We first prove the following weak convergence result:
\begin{theorem}
Let $T^N(s)=S^N(s)+I^N(s)+R^N(s)$. For any $t<\infty$, as $N\rightarrow\infty$, the stochastic model
$$
\left(\frac{S^N(s)}{T^N(s)},\frac{I^N(s)}{T^N(s)}, \frac{R^N(s)}{T^N(s)}\right)
$$ 
with initial values 
\beq
\label{CIN}
\left(\frac{S^N(0)}{N},\frac{I^N(0)}{N}, \frac{R^N(0)}{N}\right)\rightarrow (x_0,y_0,z_0)
\eeq
converges weakly to $(x_s,y_s,z_s)$ the solution of the mean-field ODE equation (\ref{MnODE}) in $[0,t]$ with initial values $(x_0,y_0,z_0)$. 
\end{theorem}

\begin{proof}
To show the weak convergence we will first use the idea of graphical representation as in \cite{DStF} to construct the inhomogenous Markov process $(S^N(t),I^N(t),R^N(t))$ from the following family of independent Poisson processes and i.i.d. uniform random variables.

[Remark]: At this point,  the graphical representation may seem
unnecessary in defining in the process. But the reason we want to use
the technique is to allow us to couple the original system and its
``truncated" version (see Table 2 for details) in the same space, and
keep them coinciding with each other for a long time. So here we
strongly recommend the reader to compare the construction below with
the one for the truncated process, to see how we can construct the two systems in the same probability space using the same family of Poisson processes and random variables.  

The construction is as follows: 
\begin{itemize}
\item For all $x\in \ZZ^+$ define a family of independent Poisson processes at rate $\nu$, denoted by $\{B_n^x: x\in \ZZ^+, n\ge 1\}$. At each space time point $(x,B_n^x)$, i.e., the $n$th jumping time of the Poisson process associate with point $x$, we have $S^N=S^N+1$ if and only if $T^N(B_n^x-)=S^N(B_n^x-)+I^N(B_n^x-)+R^N(B_n^x-)\ge x$. 
\item For all $x\in \ZZ^+$ define a family of independent Poisson processes at rate $\nu$, denoted by $\{DS_n^x: x\in \ZZ^+, n\ge 1\}$ ,which are also independent to $\{B_n^x\}$. At each space time point $(x,DS_n^x)$, we have $S^N=S^N-1$ if and only if $S^N(DS_n^x-)\ge x$. 
\item For all $x\in \ZZ^+$ define a family of independent Poisson processes at rate $\nu$, denoted by $\{DI_n^x: x\in \ZZ^+, n\ge 1\}$ which are also independent to processes defined above. At each space time point $(x,DI_n^x)$, we have $I^N=I^N-1$ if and only if $I^N(DI_n^x-)\ge x$. 
\item For all $x\in \ZZ^+$ define a family of independent Poisson processes at rate $\nu$, denoted by $\{DR_n^x: x\in \ZZ^+, n\ge 1\}$ which are also independent to processes defined above. At each space time point $(x,DR_n^x)$, we have $R^N=R^N-1$ if and only if $I^N(DR_n^x-)\ge x$.
\item For all $x\in \ZZ^+$ define a family of independent Poisson processes at rate $\gamma$, denoted by $\{RI_n^x: x\in \ZZ^+, n\ge 1\}$ which are also independent to processes defined above. At each space time point $(x,RI_n^x)$, we have $I^N=I^N-1$ and $R^N=R^N+1$ if and only if $I^N(RI_n^x-)\ge x$. 
\item For all $x\in \ZZ^+$ define a family of independent Poisson processes at rate $\beta_0(1+\beta_1)$,  denoted by $\{IN_n^{x}: x\in \ZZ^+, n\ge 1\}$,  which are also independent to processes defined above. Moreover, define a family of i.i.d. random variables $\{U_n^{x}\sim U(0,1): x\in Z^+, n\ge 1\}$ which are also independent to processes defined above. At each space-time point $[x,IN_n^x]$,  we have $S^N=S^N-1, I^N=I^N+1$ if and only if that $S^N(IN_n^x-)\ge x$ and 
$$
U_n^{x}\le \frac{\beta(IN_n^{x})I^N(IN_n^{x}-)}{\beta_0(1+\beta_1)T^N(IN_n^{x}-)}. 
$$
\end{itemize}
To show the construction above is well defined and is actually the seasonal SIR model we want, we first refer to \cite{DStF} to show the process never explodes. To show this, consider a monotone increasing process $M_t$ defined as follows:
\begin{itemize}
\item At each space time point $(x,B_n^x)$, $M_t=M_{B_n^x-}+1$ if and only if $M_{B_n^x-}\ge x$. 
\end{itemize}
Then by definition it is easy to see that $M_t\ge S^N(t)+I^N(t)+R^N(t)$ for all $t\ge 0$. Noting that for all $k$, 
$$
\sum_{i=k}^\infty\frac{1}{i}=\infty,
$$
so the bigger process $M_t$ never explodes. This implies at each time
$t$, all the jumps in our system must be in the following finite
family of Poisson processes: $\{B_n^x\}_{x=1}^{M_t},
\{DS_n^x\}_{x=1}^{M_t}, \{DI_n^x\}_{x=1}^{M_t},
\{DR_n^x\}_{x=1}^{M_t}, \{RI_n^x\}_{x=1}^{M_t}$,  and
$\{I_n^x\}_{x=1}^{M_t}$. Note that a Poisson processes (with
probability one) has only finite jumps in a finite time interval, it
is straightforward to verify that the process we defined above never
explodes. Then for each time $t$, one can easily check the transition
rates given the current value of $(S,I,R)$ in the process we construct
in the bullet list above. According to the Kolmogorov forward equation and Theorem 1 about Poisson thinning in \cite{LS}, it is easy to see that the process defined by the devices above has the same transition rates as in Table 1. Thus the process defined above is exactly the stochastic seasonal SIR model $(S^N(t),I^N(t),R^N(t))$ that we want.  

Note that, because the population size is unbounded, the transition rates of the system above can possibly (but not likely) go large. Our next step is to introduce a truncated version of the seasonal SIR model. Consider $(\hat S^N(t),\hat I^N(t),\hat R^N(t))$ to be the truncated version with new transition rates as follows: 
\begin{table}[H]
\center
\begin{tabular}{|c|c|c|c|}
\hline 
Event & Transition & Rate at which event occurs \\
\hline
Birth & $\hat S\rightarrow \hat S+1$ & $\nu [(\hat S+\hat I+\hat R)\wedge 2N]$\\
\hline
Susceptible Death  & $\hat S\rightarrow \hat S-1$ & $\nu (\hat S\wedge 2N)$\\
\hline
Infection & $\hat S\rightarrow \hat S-1, \  \hat I\rightarrow \hat I+1$ & $\beta(t)(\hat S\wedge 2N)\hat I/(\hat S+\hat I+\hat R)$\\
\hline
Recovery & $\hat I\rightarrow \hat I-1, \ \hat R\to \hat R+1$ & $\gamma (\hat I\wedge 2N)$\\
\hline
Infectious Death & $\hat I\rightarrow \hat I-1$ & $\nu (\hat I\wedge 2N)$\\
\hline
Recovered Death & $\hat R\rightarrow \hat R-1$ & $\nu (\hat R\wedge 2N)$\\
\hline
\end{tabular}
\caption{Transition rates for the truncated model}
\end{table}
\noindent Here, $a\wedge b = \min(a,b)$.  

By definition, the transition rate of  $(\hat S^N(t),\hat I^N(t),\hat R^N(t))$ is no larger than 
$$
MN=\max\{2\nu N, 2\beta_0(1+\beta_1)N, 2\gamma N\}
$$
and thus is bounded. Moreover, using the same family of Poisson processes and uniform random variables in the bullet list above, we can construct a copy of the truncated process as follows: 

\begin{itemize}
\item  At each space time point $(x,B_n^x)$, we have $\hat S^N=\hat S^N+1$ if and only if $\hat T^N(B_n^x-)=\hat S^N(B_n^x-)+\hat I^N(B_n^x-)+\hat R^N(B_n^x-)\ge x$ and $x\le 2N$. 
\item  At each space time point $(x,DS_n^x)$, we have $\hat S^N=\hat S^N-1$ if and only if $\hat S^N(DS_n^x-)\ge x$ and $x\le 2N$. 
\item  At each space time point $(x,DI_n^x)$, we have $\hat I^N=\hat I^N-1$ if and only if $\hat I^N(DI_n^x-)\ge x$ and $x\le 2N$.
\item  At each space time point $(x,DR_n^x)$, we have $\hat R^N=\hat R^N-1$ if and only if $\hat R^N(DR_n^x-)\ge x$ and $x\le 2N$.
\item  At each space time point $(x,RI_n^x)$, we have $\hat I^N=\hat I^N-1$ and $\hat R^N=\hat R^N+1$ if and only if $\hat I^N(RI_n^x-)\ge x$ and $x\le 2N$.
\item  At each space-time point $[x,IN_n^x]$, we have $\hat S^N=\hat S^N-1, \hat I^N=\hat I^N+1$ if and only if that $\hat S^N(IN_n^x-)\ge x$, $x\le 2N$ and
$$
U_n^{x}\le \frac{\beta(IN_n^{x})\hat I^N(IN_n^{x}-)}{\beta_0(1+\beta_1)\hat T^N(IN_n^{x}-)}.
$$
\end{itemize}

From the gadgets as above, it is easy to check that we defined the truncated process $(\hat S^N(t),\hat I^N(t),\hat R^N(t))$ in the same probability space as the original process $(S^N(t),I^N(t),R^N(t))$. Moreover, consider the stopping time 
\beq
\label{STD}
\tau_N=\inf\{t: T^N(t)>2N\}. 
\eeq
Then by definition we immediately have 
\begin{prop}
$(\hat S^N(t),\hat I^N(t),\hat R^N(t))\equiv (S^N(t),I^N(t),R^N(t))$ on $[0,\tau_N)$. 
\end{prop}
The next step we will show is for any given $t_0<\infty$, the total size of the population will stay near the initial value with high probability by time $t_0$ so the truncated process will, with high probability, stay together with the original one, when $N$ is large. 
\begin{lemma} \label{Stop}
For any $\ep>0$, define the stopping time $\tau_N^\ep$ to be the first time the total population size is changed by $\ep N$, i.e.
\beq
\label{STDep}
\tau_N^\ep=\inf\{t: |T^N(t)-N|>\ep N\}.
\eeq
Then for any $t_0<\infty$
\beq
\label{PST}
\lim_{N\rightarrow\infty}P(\tau^\ep_N>t_0)\rightarrow1. 
\eeq
\end{lemma}
\begin{proof}
Note that $T^N(t)$ itself also forms a Markov process with transition rates 
$$
T\to T\pm 1 \textrm{ at rate } \nu T. 
$$
So for $\Sigma^N(t)=T^N(t)/N$, we have $\Sigma^N(0)=1$ and let $Q_N(x,\cdot)$ be the measure given by the transition rate, i.e.
$$
Q_N(x,\cdot)=Nx\cdot\delta\left(x+\frac{1}{N}\right)+Nx\cdot\delta\left(x-\frac{1}{N}\right). 
$$
Then we can define
\begin{align*}
& a^N(x)=\int_{|y-x|<1} (y-x)^2Q_N(x,dy)\\
& b^N(x)=\int_{|y-x|<1} (y-x)Q_N(x,dy)\\
& \Delta_\ep^N(x)=Q_N\left(x, B^c(x,\ep)\right), 
\end{align*}
and
$$
a(x)= b(x)=0. 
$$
Then it is straightforward to check that for any $\ep>0$ and $R_0<\infty$, 
\begin{enumerate}[(i)]
\item $\lim_{N\rightarrow\infty}\sup_{|x|\le R_0} |a^N(x)-a(x)|=0$.
\item $\lim_{N\rightarrow\infty}\sup_{|x|\le R_0} |b^N(x)-b(x)|=0$.
\item $\lim_{N\rightarrow\infty}\sup_{|x|\le R_0}  \Delta_\ep^N(x)=0$.
\end{enumerate}
Thus according to Theorem 7.1 in Section 8.7 of \cite{DSD}, we have that as $N\rightarrow \infty$
\beq
\label{WC1}
\Sigma^N(t)\Rightarrow y(t) 
\eeq
and $y(t)$ is the solution of
$$
dy=a(y)dt+b(y)dB_t=0
$$
with $y(0)=1$, which is $y(t)\equiv 1$. And the proof is complete.

\end{proof}

Lemma \ref{Stop} allows us to concentrate on the truncated process $(\hat S^N(t),\hat I^N(t),\hat R^N(t))$, where the transition rates are bounded. Now consider a twice continuously differentiable function $f$ on $\RR$ such that $f(x)=x$ on $[0,3]$, $f(x)\equiv -0.5$ on $(-\infty,-1]$ and $f(x)\equiv 3.5$ on $[4,\infty)$. Then let
$$
f_1(a_1,a_2,a_3)=f\left(\frac{a_1}{N}\right). 
$$
Similarly, we can also define
$$
f_2(a_1,a_2,a_3)=f\left(\frac{a_2}{N}\right)$$
and 
$$
f_3(a_1,a_2,a_3)=f\left(\frac{a_3}{N}\right)$$
First we note that  $f_1$, $f_2$ and $f_3$ are all bounded twice continuously differentiable functions in $\RR^3$ and that 
\beq
\label{FS}
f_1\left(\hat S^N(t),\hat I^N(t),\hat R^N(t)\right)=\frac{\hat S^N(t)}{N}=\frac{S^N(t)}{N}
\eeq
on $[0,\tau_N)$. Thus, using the inhomogenous Dynkin's formula (see, for example, Section 7.3 of \cite{FBH}), we have 
\small
\beq
\begin{aligned}
M_t=&f_1\left(\hat S^N(t),\hat I^N(t),\hat R^N(t)\right)-f_1\left(\hat S^N(0),\hat I^N(0),\hat R^N(0)\right)\\
&-\int_0^t \hat{\mathcal{A}}_s\left[f_1\left(\hat S^N(s),\hat I^N(s), \hat R^N(s)\right)\right]ds
\end{aligned}
\eeq
\normalsize
is a martingale with mean 0. Here $\hat{\mathcal{A}}_s\left[f_1\left(\hat S^N(s),\hat I^N(s), \hat R^N(s)\right)\right]$ is the infinitesimal generator of the truncated process, applying on $f_1$, i.e. 
\beq
\label{InfgenH}
\begin{aligned}
\hat{\mathcal{A}}_s&\left[f_1\left(\hat S^N(s),\hat I^N(s), \hat R^N(s)\right)\right]\\
&=\hat\lambda_b\left[f_1\left(\hat S^N(s)+1,\hat I^N(s), \hat R^N(s)\right)-f_1\left(\hat S^N(s),\hat I^N(s), \hat R^N(s)\right)\right]\\
&+\hat\lambda_{ds}\left[f_1\left(\hat S^N(s)-1,\hat I^N(s), \hat R^N(s)\right)-f_1\left(\hat S^N(s),\hat I^N(s), \hat R^N(s)\right)\right]\\
&+\hat\lambda_{in}(s)\left[f_1\left(\hat S^N(s)-1,\hat I^N(s)+1, \hat R^N(s)\right)-f_1\left(\hat S^N(s),\hat I^N(s), \hat R^N(s)\right)\right]\\
&+\hat\lambda_{ri}\left[f_1\left(\hat S^N(s),\hat I^N(s)-1, \hat R^N(s)+1\right)-f_1\left(\hat S^N(s),\hat I^N(s), \hat R^N(s)\right)\right]\\
&+\hat\lambda_{di}\left[f_1\left(\hat S^N(s),\hat I^N(s)-1, \hat R^N(s)\right)-f_1\left(\hat S^N(s),\hat I^N(s), \hat R^N(s)\right)\right]\\
&+\hat\lambda_{dr}\left[f_1\left(\hat S^N(s),\hat I^N(s), \hat R^N(s)-1\right)-f_1\left(\hat S^N(s),\hat I^N(s), \hat R^N(s)\right)\right]
\end{aligned}
\eeq
where $\hat\lambda_b, \hat\lambda_{ds},  \hat\lambda_{in},  \hat\lambda_{ri}, \hat\lambda_{di}$ and $\hat\lambda_{dr}$
are the transition rates of $\left(\hat S^N(s),\hat I^N(s), \hat R^N(s)\right)$ given by the corresponding entires of table 2. 

Then according to Lemma 3.3 in \cite{DZ}, it is straightforward to show that $M_t$ is a martingale of finite variation, so that it has quadratic variation:
$$
[M]_{t_0}=\sum_{t\in \Pi_{t_0}} \left[f_1\left(\hat S^N(t),\hat I^N(t),\hat R^N(t)\right)-f_1\left(\hat S^N(t^-),\hat I^N(t^-),\hat R^N(t^-)\right)\right]^2
$$
where $\Pi_{t_0}$ is the the set of jumping times before time $t_0$. Thus there exists some constant $M_0$ such that for any $t>0$
$$
E(M_{t_0}^2)=E[M]_{t_0}\le \frac{2M_0}{N}
$$
and hence
\beq
\label{BML2}
E\left( \sup_{0\le s\le t} M_s^2\right)\le \frac{8M_0}{N},
\eeq
which implies for any $\ep\in (0,1/2)$ and $t_0<\infty$
\beq
\label{NPa}
\lim_{N\rightarrow \infty}P(|M_s|<\ep \textrm{ for all } s\in [0,t_0])=1.
\eeq
Consider the event
$$
A_N^{(1)}=\{(|M_s|<\ep \textrm{ for all } s\in [0,t_0]\}\cap \{\tau_N^\ep>t_0\}. 
$$
By equations (\ref{FS}) and (\ref{InfgenH}),  we have for any path in $A_N^{(1)}$ and any time $s< t_0$,
$$
\begin{aligned}
\hat{\mathcal{A}}_s&\left[f_1\left(\hat S^N(s),\hat I^N(s), \hat R^N(s)\right)\right]\\
&=\nu\frac{\left[S^N(s)+I^N(s)+R^N(s)\right]}{N}-\frac{\nu S^N(s)}{N}- \beta(s) \frac{I^N(s) S^N(s)}{N\left[S^N(s)+I^N(s)+R^N(s)\right]}
\end{aligned}
$$
which implies 
$$
\begin{aligned}
&\left|\hat{\mathcal{A}}_s\left[f_1\left(\hat S^N(s),\hat I^N(s), \hat R^N(s)\right)\right]-F_1\left(S^N(s),I^N(s), R^N(s)\right)\right| \le 8\beta(1+\beta_1)\ep
\end{aligned}
$$
for all paths in $A_N^{(1)}$ and times $s< t_0$. So there is $C_0=8\beta(1+\beta_1)\ep t_0+1$ such that, given the event $A_N^{(1)}$, 
\beq
\label{ODE1}
\sup_{0\le t\le t_0}\left|\frac{S^N(t)}{N}-\frac{S^N(0)}{N}-\int_0^t F_1\left(S^N(s),I^N(s), R^N(s)\right) ds\right|<C_0\ep. 
\eeq
Repeat exactly the same process as above with $f_2$ and $f_3$, we similarly have high probability events $A^{(2)}_N$ and $A^{(3)}_N$ such that under $A^{(2)}_N$
\beq
\label{ODE2}
\sup_{0\le t\le t_0}\left|\frac{I^N(t)}{N}-\frac{I^N(0)}{N}-\int_0^t F_2\left(S^N(s),I^N(s), R^N(s)\right)ds\right|<C_0\ep
\eeq
and under  $A^{(3)}_N$
\beq
\label{ODE3}
\sup_{0\le t\le t_0}\left|\frac{R^N(t)}{N}-\frac{R^N(0)}{N}-\int_0^t F_3\left(S^N(s),I^N(s), R^N(s)\right) ds\right|<C_0\ep.
\eeq
Then consider the following high probability event $A_N=A^{(1)}_N\cap A^{(2)}_N\cap A^{(3)}_N$. Combining inequalities 
(\ref{ODE1}), (\ref{ODE2}) and (\ref{ODE3}), and noting that the derivatives $F_1$, $F_2$ and $F_3$ are Lipschitz functions on
$\RR^+\times[0,3]^3$, then a standard ODE argument (see Theorem (2.11) in \cite{KZ}, for example) shows that there is some $C_1<\infty$ such that 
$$
\left| \frac{S^N(t)}{N}- x_t\right|\le\left( \left| \frac{S^N(0)}{N}- x_0\right|+C_0\ep\right) e^{C_1t},
$$
$$
\left| \frac{I^N(t)}{N}- y_t\right|\le\left( \left| \frac{I^N(0)}{N}- y_0\right|+C_0\ep\right) e^{C_1t},
$$
and that
$$
\left| \frac{R^N(t)}{N}- z_t\right|\le\left( \left| \frac{R^N(0)}{N}- z_0\right|+C_0\ep\right) e^{C_1t}
$$
for all $t\in [0,t_0]$. Moreover, recalling the definition of $\tau_N^\ep$, for all paths in $A_N$ and any $t\in[0,t_0]$, we have 
\begin{align*}
\left| \frac{S^N(t)}{T^N(t)}- x_t\right|&\le \left| \frac{S^N(t)}{N}- x_t\right|+\left| \frac{S^N(t)(T^N(t)-N)}{T^N(t)N}\right| \\
&= \left| \frac{S^N(t)N\ep}{T^N(t)N}\right|+\left( \left| \frac{S^N(0)}{N}- x_0\right|+C_0\ep\right) e^{C_1t} \\
&\le \ep+\left( \left| \frac{S^N(0)}{N}- x_0\right|+C_0\ep\right) e^{C_1t}. 
\end{align*}
Similarly,
$$
\left| \frac{I^N(t)}{T^N(t)}- y_t\right|\le\ep+\left( \left| \frac{I^N(0)}{N}- y_0\right|+C_0\ep\right) e^{C_1t},
$$
and
$$
\left| \frac{R^N(t)}{T^N(t)}- z_t\right|\le\ep+\left( \left| \frac{R^N(0)}{N}- z_0\right|+C_0\ep\right) e^{C_1t}. 
$$
Since $\ep$ can be arbitrarily small and because of equation (\ref{CIN}), the proof of Theorem 1 is complete. 
\end{proof}

Form the calculations above, one immediately has the following corollary:
\begin{corollary}
Consider $\xi^N_t=(S^N(t),I^N(t), R^N(t))/N\in [(\ZZ^+\cup \{0\})/N]^3$. Then with initial value
$$
\label{CIN}
\left(\frac{S^N(0)}{N},\frac{I^N(0)}{N}, \frac{R^N(0)}{N}\right)\rightarrow (x_0,y_0,z_0),
$$
and any $t>0$, $\xi^N_t$ converges weakly to $\xi_s=(x_s,y_s,z_s)$ the solution of the mean-field ODE equation (\ref{MnODE}) in $[0,t]$ with initial values $(x_0,y_0,z_0)$. 
\end{corollary}

\section{A Central Limit Theorem} 
In this section, our goal is to prove a central limit theorem showing that $\xi^N_t$ minus its drift part converges weakly to a diffusion process after proper scaling. We will begin with several notions: for any $(x,y,z)\in [(\ZZ^+\cup \{0\})/N]^3$ and any $t\ge 0$, let $\lambda^{(N)}_t(x,y,z)$ be the total transition rate at time $t$ where the configuration of $\xi^N_t$ equals $(x,y,z)$. By definition it is straightforward to see that 
\beq
\label{TrRate}
\lambda^{(N)}_t(x,y,z)=2N\nu(x+y+z)+\frac{N\beta(t)xy}{x+y+z}+N\gamma y.
\eeq
Then, let $\mu^{(N)}_t(x,y,z)$ be the outcome distribution after a transition at time $t$ given $\xi^N_{t-}=(x,y,z)$. One can easily show that 
\beq
\label{TrDist}
\begin{aligned}
\mu^{(N)}_t(x,y,z)&=\frac{N\nu(x+y+z)}{\lambda^{(N)}_t(x,y,z)}\delta\left(x+N^{-1},y,z\right)+\frac{N\nu x}{\lambda^{(N)}_t(x,y,z)}\delta\left(x-N^{-1},y,z\right)\\
&+\frac{N\nu y}{\lambda^{(N)}_t(x,y,z)}\delta\left(x,y-N^{-1},z\right)+\frac{N\nu z}{\lambda^{(N)}_t(x,y,z)}\delta\left(x,y,z-N^{-1}\right)\\
&+\frac{N\beta(t)xy}{(x+y+z)\lambda^{(N)}_t(x,y,z)}\left(x-N^{-1},y+N^{-1},z\right)\\
&+\frac{N\gamma}{\lambda^{(N)}_t(x,y,z)}\left(x,y-N^{-1},z+N^{-1}\right).
\end{aligned}
\eeq
Then for all $\vec x\in  [(\ZZ^+\cup \{0\})/N]^3$, we define
\beq
\label{InfGen}
F_N(\vec x,t)=\lambda^{(N)}_t(\vec x) \int (\vec z-\vec x)d\mu^{(N)}_t(\vec x, d\vec z)
\eeq
which, by definition, can be written explicitly as 
$$
F_N(\vec x,t)=F(\vec x,t)=\left(
\begin{array} {c}
\nu(y+z)- \frac{\beta(t)xy}{(x+y+z)} \\
\frac{\beta(t)xy}{(x+y+z)}-(\nu+\gamma) y  \\
\gamma y- \nu z 
\end{array}
\right)^T
$$
which is independent of $N$. Similarly, for any $\vec x\in  [(\ZZ^+\cup \{0\})/N]^3$ and $i,j\in \{1,2,3\}$, define 
\beq
\label{InfCov}
g^{(N)}_{i,j}(\vec x, t)=\lambda^{(N)}_t(\vec x) \int (z_i-x_i)(z_j-x_j)d\mu^{(N)}_t(\vec x, d\vec z)
\eeq
and
$$
G^{(N)}(\vec x, t)=\left( g^{(N)}_{i,j}(\vec x, t)\right)_{3\times 3}
$$
which is the infinitesimal covariance matrix of the system. It is easy to see that the matrix $G^{(N)}(\vec x, t)$ can be written explicitly as
$$
G^{(N)}(\vec x, t)=\left(
\begin{array} {ccc}
\frac{\nu(2x+y+z)}{N}+\frac{\beta(t)xy}{N(x+y+z)} & -\frac{\beta(t)xy}{N(x+y+z)}                                          &0\\ 
-\frac{\beta(t)xy}{N(x+y+z)}                                   &\frac{\beta(t)xy}{N(x+y+z)}+\frac{(\nu+\gamma) y}{N}   &-\frac{\gamma y}{N}\\
0                                                                           &-\frac{\gamma y}{N}                                                         &\frac{\mu z+\gamma y}{N}
\end{array}
\right).
$$
Note that for any $N$, $NG^{(N)}(\vec x, t)\equiv G(\vec x, t)=\left( g_{i,j}(\vec x, t)\right)_{3\times 3}$ where 
$$
G(\vec x, t)=\left(
\begin{array} {ccc}
\nu(2x+y+z)+\frac{\beta(t)xy}{(x+y+z)} & -\frac{\beta(t)xy}{(x+y+z)}                                          &0\\ 
-\frac{\beta(t)xy}{(x+y+z)}                                   &\frac{\beta(t)xy}{(x+y+z)}+(\nu+\gamma) y   &-\gamma y\\
0                                                                           &-\gamma y                                                       &\mu z+\gamma y
\end{array}
\right).
$$
The following result shows that the process minus the drift part converges weakly to a diffusion after proper scaling. 
\begin{theorem}
\label{ConDif}
Define the stochastic process
$$
W_N(t)=\sqrt{N}\left(\xi^N_t-\xi^N_0-\int_0^t F(\xi^N_s,s)ds\right).
$$
Then as $N\rightarrow\infty$, $W_N(t)$ converges to the diffusion
process $W(t)$ with the following characteristic function: for all $\vec\theta=(\theta_1,\theta_2,\theta_3)$
$$
\phi(t,\vec\theta)=E(\exp\left[i\vec\theta \cdot W(t)\right])=\exp\left[ -\frac{1}{2}\sum_{i,j}\theta_i\theta_j\int_0^t g_{i,j}(\xi_s,s)ds\right].
$$ 
\end{theorem}
[Remark]: The theorem above implies that for all $t\ge 0$ $W_N(t)$ converges weakly to the 3 dimensional normal distribution with mean 0 and characteristic function given by $\phi(t,\vec\theta)$. This explains why we call this section a central limit theorem. 

\begin{proof}
In order to bound the drift and transition rates, we again need to consider the truncated process. Let 
$$
\hat\xi^N_t=(\hat S^N(t),\hat I^N(t), \hat R^N(t))/N\in [(\ZZ^+\cup \{0\})/N]^3.
$$
Similarly, we can define 
\begin{align*}
\hat\lambda^{(N)}_t(x,y,z)&=N\nu[(x+y+z)\wedge 2]+N\nu(x\wedge 2+y\wedge 2+z\wedge 2)\\
&+\frac{N\beta(t)(x\wedge 2)y}{x+y+z}+N\gamma (y\wedge 2)
\end{align*}
to be the transition rate and let 
$$
\begin{aligned}
\hat\mu^{(N)}_t(x,y,z)&=\frac{N\nu[(x+y+z)\wedge 2]}{\lambda^{(N)}_t(x,y,z)}\delta\left(x+N^{-1},y,z\right)+\frac{N\nu (x\wedge 2)}{\lambda^{(N)}_t(x,y,z)}\delta\left(x-N^{-1},y,z\right)\\
&+\frac{N\nu (y\wedge 2)}{\lambda^{(N)}_t(x,y,z)}\delta\left(x,y-N^{-1},z\right)+\frac{N\nu (z\wedge 2)}{\lambda^{(N)}_t(x,y,z)}\delta\left(x,y,z-N^{-1}\right)\\
&+\frac{N\beta(t)(x\wedge 2)y}{(x+y+z)\lambda^{(N)}_t(x,y,z)}\left(x-N^{-1},y+N^{-1},z\right)\\
&+\frac{N(y\wedge 2)\gamma}{\lambda^{(N)}_t(x,y,z)}\left(x,y-N^{-1},z+N^{-1}\right).
\end{aligned}
$$
Then we can define the drift 
\beq
\label{InfGen}
\hat F_N(\vec x,t)=\hat F(\vec x,t)=\hat\lambda^{(N)}_t(\vec x) \int (\vec z-\vec x)d\hat\mu^{(N)}_t(\vec x, d\vec z),
\eeq
where it is easy to check that 
$$
\hat F(\vec x,t)=\left(
\begin{array} {c}
\nu[(x+y+z)\wedge 2]-\nu(x\wedge 2)- \frac{\beta(t)(x\wedge 2)y}{(x+y+z)} \\
\frac{\beta(t)(x\wedge 2)y}{(x+y+z)}-(\nu+\gamma) (y\wedge 2)  \\
\gamma (y\wedge 2)- \nu (z\wedge 2)
\end{array}
\right)^T. 
$$ 
and 
$$
\hat W_N(t)=\sqrt{N}\left(\hat\xi^N_t-\hat \xi^N_0-\int_0^t F(\hat \xi_s,s)ds\right).
$$

From the discussion in Section 3,  it is easy to see that for any $t_0<\infty$, under the event $A_N$, we have that $\hat W_N(t)\equiv W_N(t)$ for all $t\le t_0$. According to Lemma \ref{Stop}, it is easy to see that 
\beq
\label{DRTr}
\hat W_N(t)-W_N(t)\rightarrow 0
\eeq
in probability as $N\rightarrow\infty$, which implies that in
order to prove Theorem 2, it suffices to show that $\hat W_N(t)\Rightarrow W(t)$ as $N\rightarrow\infty$. Moreover, for any $\vec x\in  [(\ZZ^+\cup \{0\})/N]^3$ and $i,j\in \{1,2,3\}$, we can similarly define 
\beq
\label{InfCov}
\hat g^{(N)}_{i,j}(\vec x, t)=\hat\lambda^{(N)}_t(\vec x) \int (z_i-x_i)(z_j-x_j)d\hat\mu^{(N)}_t(\vec x, d\vec z)
\eeq
and
$$
\hat G^{(N)}(\vec x, t)=\left(\hat g^{(N)}_{i,j}(\vec x, t)\right)_{3\times 3}
$$
which is the infinitesimal covariance matrix of the system. Similarly to before, the matrix $\hat G^{(N)}(\vec x, t)$ can be written explicitly as
$$
\hat G^{(N)}(\vec x, t)=\left(
\begin{array} {ccc}
\frac{\nu[(x+y+z)\wedge 2+x\wedge 2]}{N}+\frac{\beta(t)(x\wedge 2)y}{N(x+y+z)} & -\frac{\beta(t)(x\wedge 2)y}{N(x+y+z)}                                          &0\\ 
-\frac{\beta(t)(x\wedge 2)y}{N(x+y+z)}                                   &\frac{\beta(t)(x\wedge 2)y}{N(x+y+z)}+\frac{(\nu+\gamma) (y\wedge 2)}{N}   &-\frac{\gamma (y\wedge 2)}{N}\\
0                                                                           &-\frac{\gamma (y\wedge 2)}{N}                                                         &\frac{\nu (z\wedge 2)+\gamma (y\wedge 2)}{N}
\end{array}
\right).
$$
Note that for any $N$, $N\hat G^{(N)}(\vec x, t)\equiv \hat G(\vec x, t)=\left( \hat g_{i,j}(\vec x, t)\right)_{3\times 3}$ is equal to
$$
\left(
\begin{array} {ccc}
\nu[(x+y+z)\wedge 2+x\wedge 2]+\frac{\beta(t)(x\wedge 2)y}{(x+y+z)} & -\frac{\beta(t)(x\wedge 2)y}{(x+y+z)}                                          &0\\ 
-\frac{\beta(t)(x\wedge 2)y}{(x+y+z)}                                   &\frac{\beta(t)(x\wedge 2)y}{(x+y+z)}+(\nu+\gamma) (y\wedge 2)   &-\gamma (y\wedge 2)\\
0                                                                           &-\gamma (y\wedge 2)                                                       &\nu (z\wedge 2)+\gamma (y\wedge 2)
\end{array}
\right).
$$
Recall the definition of the mean-field ODE equation (\ref{MnODE}),
its solution $\xi_t=(\xi_t^{(1)},\xi_t^{(2)},\xi_t^{(3)})$ satisfies
that $\xi_t\in [0,1]^3$ and that
$\xi_t^{(1)}+\xi_t^{(2)}+\xi_t^{(3)}\equiv 1$. We have $ \hat G(\xi_t,
t)\equiv G(\xi_t, t)$ for all $t\ge0$. At this point, we have shown
that to show Theorem 2 it suffices to prove the following lemma that
gives the parallel result for the truncated process $\hat \xi_t$. 
\begin{lemma}
\label{LCDH}
As $N\rightarrow\infty$, $\hat W_N(t)$ converges to the diffusion $\hat W(t)= W(t)$. 
\end{lemma}
\begin{proof} Here we imitate the proof of Theorem (3.1) in \cite{KZ71}. First, the tightness of the sequence $\hat W_N(t)$ can be easily verified by inequality (\ref{BML2}) controlling the $L^2$ norm of the Dynkin's Martingale and the fact that $P(\tau_N^\ep\le t)\to 0$. For any $t\ge 0$ and any $\vec\theta=(\theta_1,\theta_2,\theta_3)$, define the characteristic function 
\beq
\label{CF1}
\phi_N(t,\vec\theta)=E\left(\exp\left[i\vec\theta\cdot W_N(t)\right] \right). 
\eeq
Here and in the rest of this section, ``$\cdot$" stands for dot product. Noting that the diffusion process is Gaussian and has independent increments, according to the proof in \cite{KZ71}, to prove this lemma, it suffices to show that 
\beq
\label{CCF}
\lim_{N\rightarrow\infty} \phi_N(t,\vec\theta)=\phi(t,\vec\theta)
\eeq
for all $t$ and $\theta$. To prove equation (\ref{CCF}), we introduce the Markov process in $(\RR^+\cup\{0\})^6$
$$
\zeta^{N}_s=\{\zeta^{N}_s(i)\}_{i=1}^6=\left(\hat\xi^{N}_s, \hat W_N(s)\right). 
$$
Recall the construction of the truncated process in Table 2: each jump in $\hat\xi^{N}_s$ corresponds to a jump in a finite family of Poisson processes. So for any $t\ge 0$, $\zeta^{N}_s$ is a bounded semimartingale on $[0,t]$ with finite variation. Moreover, it is easy to see that we have the decomposition: 
$$
\zeta^{N}_s=(\zeta^{N}_s)^c+(\zeta^{N}_s)^d
$$
where 
$$
(\zeta^{N}_s)^c=\left(\vec 0, -\sqrt{N}\int_0^s \hat F_N(\hat\xi^{N}_r,r)dr\right)
$$
is a continuous semimartingale and 
$$
(\zeta^{N}_s)^d=\left(\hat\xi^{N}_s,\sqrt{N}(\hat\xi^{N}_s-\hat\xi^{N}_0)\right)
$$
is a pure jump process. Thus for any twice continuously differentiable bounded function $f: \RR^6\rightarrow \RR$ and any $t\ge 0$, by Ito's formula, see Theorem 3.9.1 of \cite{Bicht} for example, we have
\beq
\label{ItoF1}
\begin{aligned}
f(\zeta^{N}_t)=&f(\zeta^{N}_0)+\int_0^t \nabla f(\zeta^{N}_{s-}) \cdot d\zeta^{N}_{s}\\
&+\sum_{s\le t} \left[\Delta f(\zeta^{N}_{s})-\nabla f(\zeta^{N}_{s-})\cdot \Delta\zeta^{N}_{s} \right].
\end{aligned}
\eeq
Then according to the decomposition above, and the fact that $\zeta^{N}_{s}$ with probability one is of finite variation with at most finite jumps, Remark  3.7.27 (iv) in \cite{Pro} guarantees that 
$$
\int_0^t \nabla f(\zeta^{N}_{s-}) \cdot d\zeta^{N}_{s}
$$
is an ordinary Lebesgue-Stieltjes integral pathwisely. This allow us to decompose the (finite) jumps in $\zeta^{N}_{s}: s\in[0,t]$, which cancels with last term in equation (\ref{ItoF1}). So we have: 
\beq
\label{ItoF2}
\begin{aligned}
f(\zeta^{N}_t)&=f(\zeta^{N}_0)+\int_0^t \nabla f(\zeta^{N}_{s-})\cdot d(\zeta^{N}_{s})^c+\sum_{s\le t} \Delta f(\zeta^{N}_{s})\\
&=f(\zeta^{N}_0)+\int_0^t \nabla f(\zeta^{N}_{s})\cdot[ \vec 0, -\sqrt{N}\hat F(\xi^{N}_{s},s)]ds+\sum_{s\le t} \Delta f(\zeta^{N}_{s})
\end{aligned}
\eeq
Let $f(\vec x)=h\left (x_4,x_5,x_6\right)=\exp\left[i\vec\theta\cdot (x_4,x_5,x_6)\right] $. Then we have
\beq
h\left(\hat W_N(t)\right)=1-\sqrt{N}\int_0^t \nabla h\left(\hat W_N(s)\right)\cdot \hat F(\xi^{N}_{s},s)ds+\sum_{s\le t} \Delta h\left(\hat W_N(s)\right).
\eeq
Taking expectation on both sides, and noting that all the functions $h,\hat F$ and $\hat\lambda_N$ are Lipschitz, we have that
\begin{align*}
\phi_N(t,\vec\theta)=&1-\sqrt{N} \int_0^t E\left(\nabla h \left(\hat W_N(s)\right)\cdot \hat\lambda^{(N)}_s(\hat\xi^{N}_s)\int (\vec z-\hat\xi^{N}_s) \hat\mu^{(N)}_s(\hat\xi^{N}_s,dz)\right)ds\\
&+  \int_0^t E\left(\hat\lambda^{(N)}_s(\hat\xi^{N}_s)\int \left[h\left(\hat W_N(s)+\sqrt{N}(z-\hat\xi^{N}_s)\right)- h\left(\hat W_N(s)\right)\right] \hat\mu^{(N)}_s(\hat\xi^{N}_s,dz)\right)ds
\end{align*}
which implies that  
\begin{align*}
\phi_N(t,\vec\theta)=1+E\left( \int_0^t \hat\lambda^{(N)}_s(\hat\xi^{N}_s)\int \left[T^1(h,\hat W_N(s),\hat\xi^N_s,z)\right] \hat\mu^{(N)}_s(\hat\xi^{N}_s,dz)\right) 
\end{align*}
where
$$
T^1(h,\hat W_N(s),\hat\xi^N_s,z)=h\left(\hat W_N(s)+\sqrt{N}(z-\hat\xi^{N}_s)\right)-h\left(\hat W_N(s)\right)-\sqrt{N}(z-\hat\xi^{N}_s)\cdot \nabla h \left(\hat W_N(s)\right). 
$$
Note that $\hat G$ and $h$ are both Lipschitz functions and that the function
$$
\psi(u)=\frac{e^{iu}-iu+\frac{u^2}{2}}{u^2}
$$
is Lipschitz on [-1/2,1/2]. Using the same calculation as on page 352 of \cite{KZ71}, we have
\begin{align*}
&\phi_N(t,\vec\theta)-1=-\int_0^tE\left(\frac{1}{2}\sum_{j,k\in\{1,2,3\}}\theta_j\theta_k \hat g_{j,k}(\hat\xi^N_s,s)\exp\left[i\vec\theta\cdot \hat W_N(s)\right]\right)ds\\
&+\int_0^tE\left(\exp\left[i\vec\theta\cdot \hat W_N(s)\right] \lambda^{(N)}_s(\hat\xi^{N}_s)\int \psi\left(\sqrt{N}\vec\theta\cdot(z-\xi^{N}_s)\right)\left[\sqrt{N}\vec\theta\cdot(z-\xi^{N}_s)\right]^2 \hat\mu^{(N)}_s(\hat\xi^{N}_s,dz)\right)ds.
\end{align*}
Denoting the second term on the right side as $K_N(t,\vec\theta)$,
note that the measures $\hat\mu^{(N)}_s(\hat\xi^{N}_s,\cdot)$ always
put mass one on the neighborhood $B_\infty(\hat\xi^{N}_s,N^{-1})$, $\hat G$ is bounded, and that $\psi(u)\rightarrow 0$ as $u\to 0$. It is straightforward to see that for any $t$ and $\vec\theta$, $K_N(t,\vec\theta)\to 0$ as $N\to\infty$. Thus we have
$$
\phi_N(t,\vec\theta)-1=-\int_0^tE\left(\frac{1}{2}\sum_{j,k\in\{1,2,3\}}\theta_j\theta_k \hat g_{j,k}(\hat\xi^N_s,s)\exp\left[i\vec\theta\cdot \hat W_N(s)\right]\right)ds +o(1).
$$
In Corollary 1 we proved that $\hat\xi^N_s$ converges weakly to the (deterministic) solution $\xi_s$. Then for any $\ep$ and $t_0$ under the high probability event 
$$
\hat A_N=\left\{\sup_{s\in [0,t_0]} |\hat\xi^N_s-\xi_s|<\ep\right\}
$$
noting that $\hat g_{j,k}$ are all Lipschitz functions, we have that there are some constants $K,L<\infty$, such that when $N$ is large, for all $t\le t_0$:
\begin{align*}
&\phi_N(t,\vec\theta)-1+\int_0^t\frac{1}{2}\sum_{j,k\in\{1,2,3\}}\theta_j\theta_k \hat g_{j,k}(\xi_s,s)E\left(\exp\left[i\vec\theta\cdot \hat W_N(s)\right]\right)ds \\
&\in \left[-K\ep-LP(\hat A_N^c)-o(1) ,K\ep+LP(\hat A_N^c)+o(1)\right].
\end{align*}
Thus, we have proven that for any $\ep$, when $N$ is large: 
\beq
\label{CD1}
\sup_{t\in [0,t_0]}\left|\phi_N(t,\vec\theta)-1+\int_0^t\frac{1}{2}\sum_{j,k\in\{1,2,3\}}\theta_j\theta_k \hat g_{j,k}(\xi_s,s)\phi_N(s,\vec\theta)ds\right|\le 2K\ep. 
\eeq
Thus, the desired result that $\phi_N(t,\vec\theta)\to
\phi(t,\vec\theta)$ follows directly from the same ODE argument in \cite{KZ71}, which completes the proof of this lemma. 
\end{proof}
Combining the result in Lemma \ref{LCDH} and the fact in equation (\ref{DRTr}) that $\hat W_N(t)-W_N(t)\Rightarrow 0$, the proof of Theorem 2 is complete. 
\end{proof}

\section{Simulation Results}

In this section we present results of some numerical simulations that
illustrate the above theory. The four subgraphs (a)- (d) in Figure 1 show single
realizations of the model for four different initial population sizes,
together with the deterministic solution of the model. In each case we
show the fraction of each type against time. As predicted by the theory,
realizations of the stochastic model fluctuate around the
deterministic solution, with noticeably smaller fluctuations for
larger population sizes.
\begin{figure} \centering 
\subfigure[] {
\includegraphics[width=0.475\columnwidth]{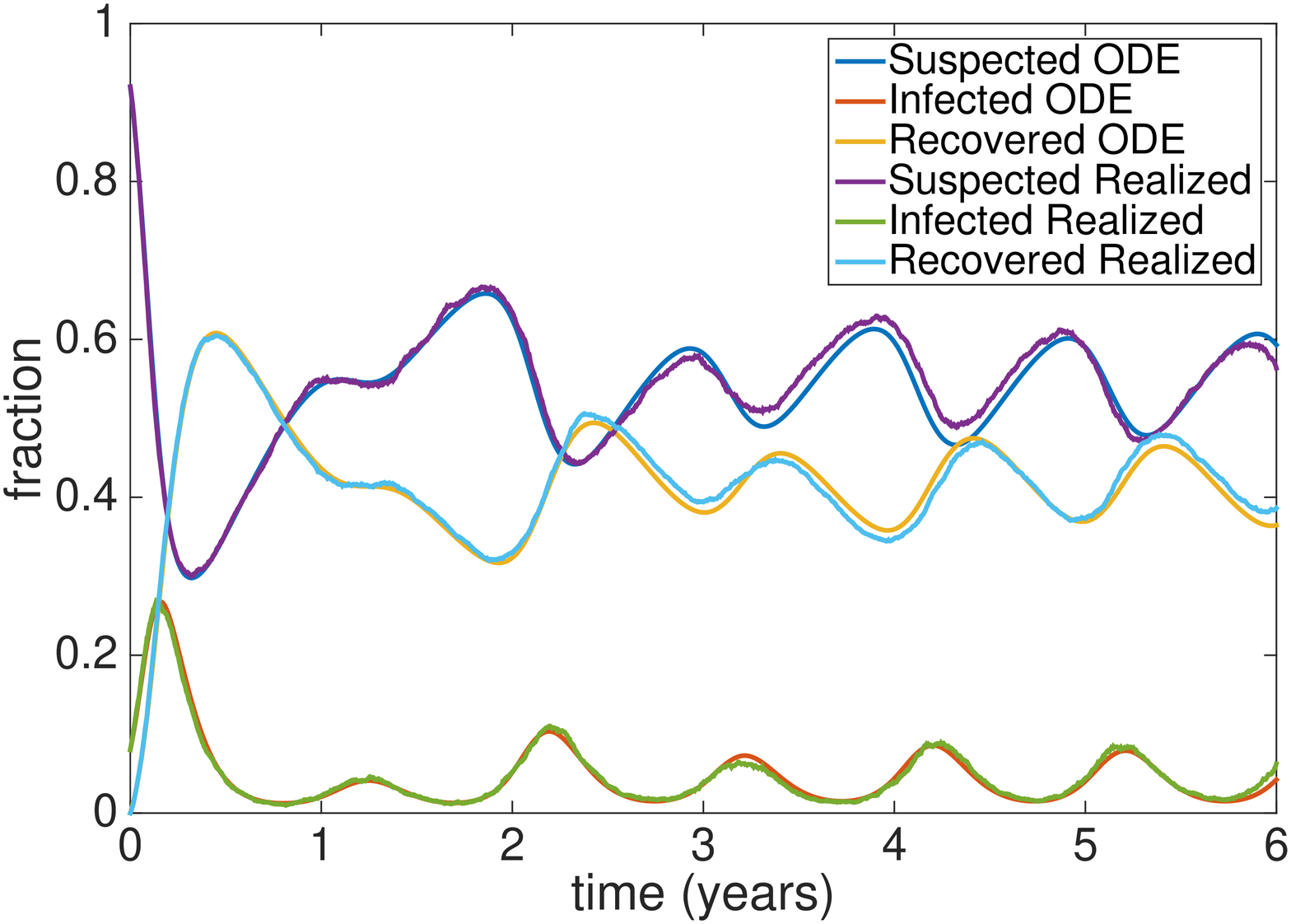} 
} 
\subfigure[] {
\includegraphics[width=0.475\columnwidth]{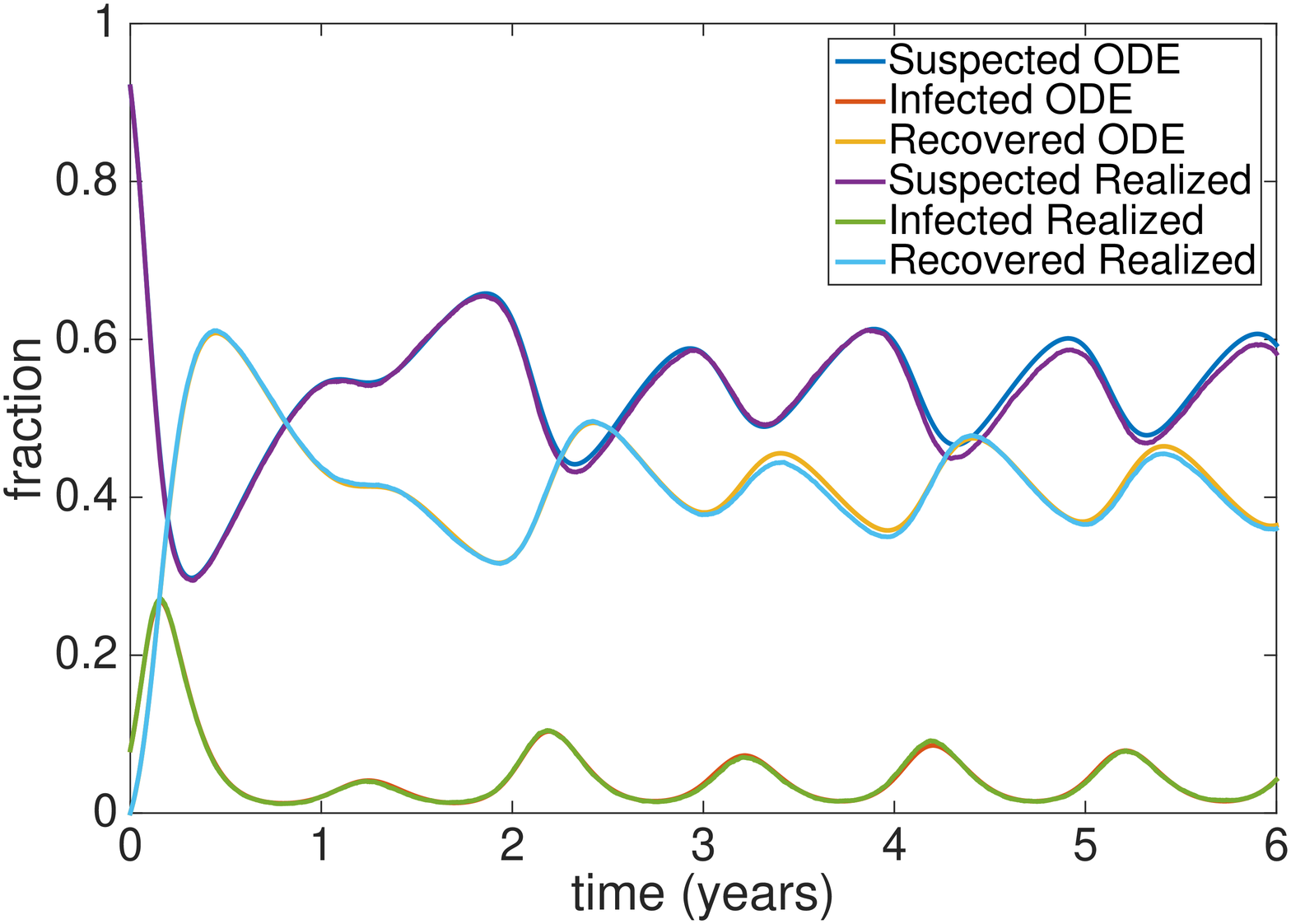} 
} 
\subfigure[] { 
\includegraphics[width=0.475\columnwidth]{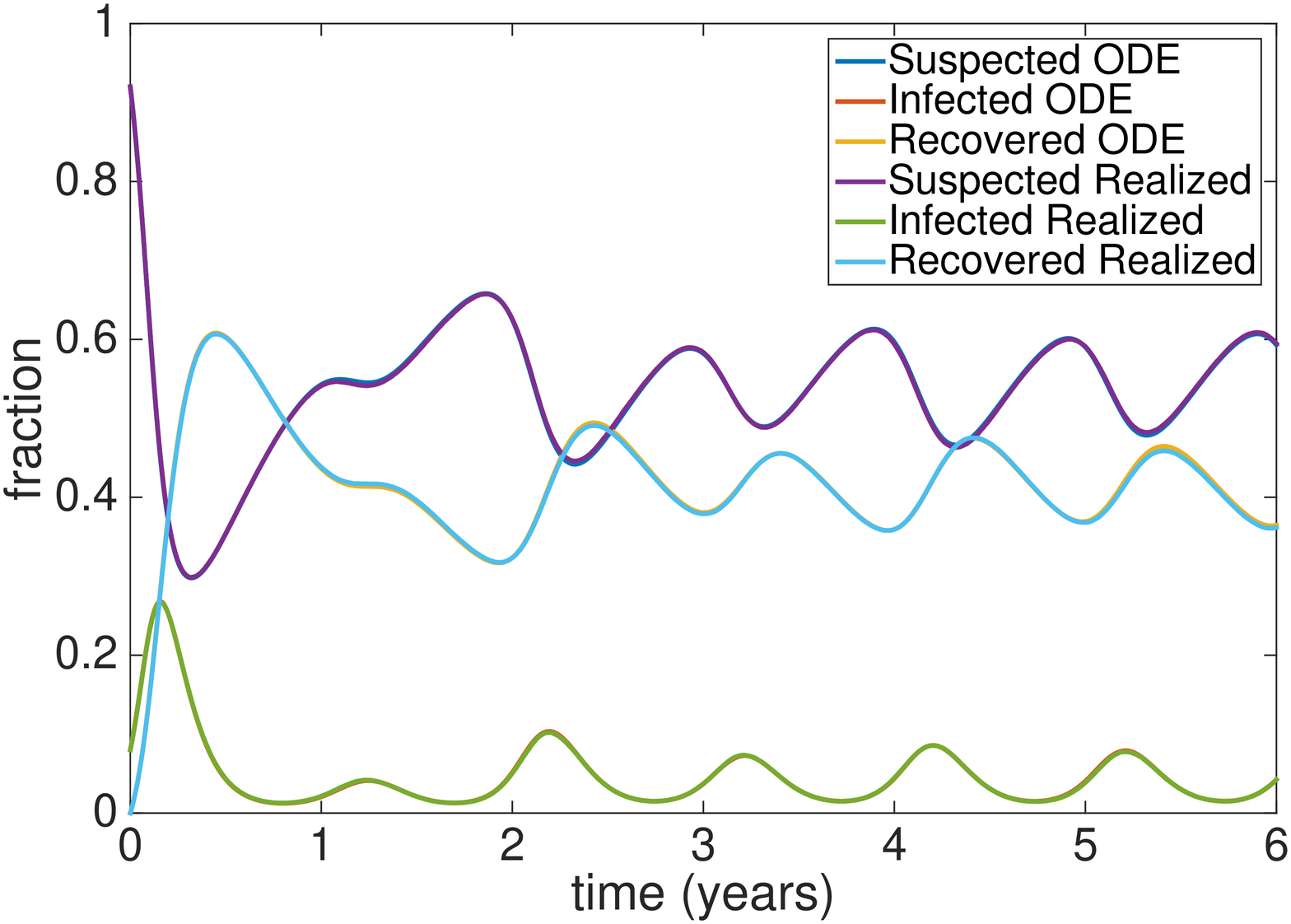} 
} 
\subfigure[] {
\includegraphics[width=0.475\columnwidth]{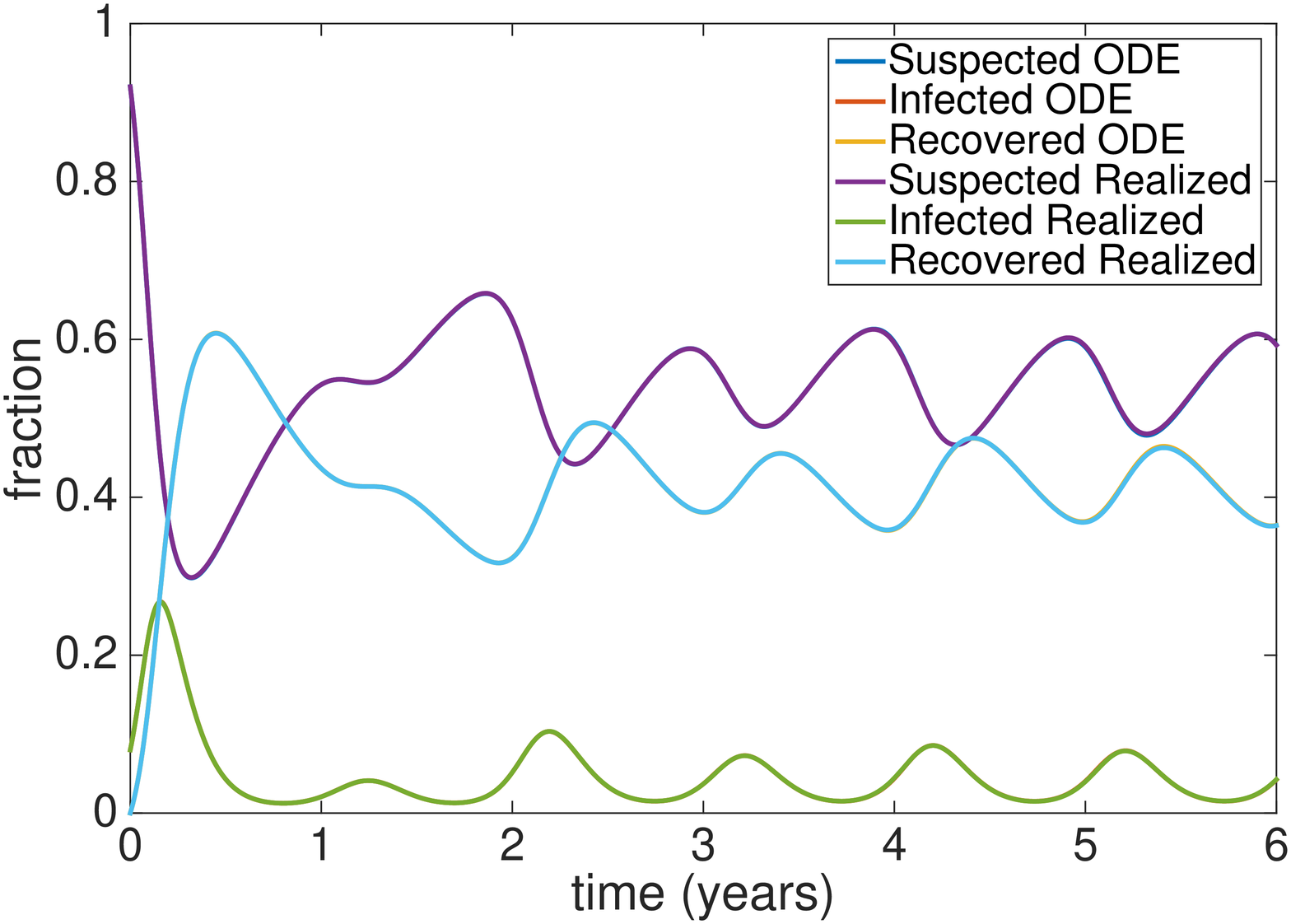} 
} 
\label{fig} 
 \caption{Single realizations of the seasonally forced epidemic
      model for initial population sizes (a) $N=10^4$, (b) $N=10^5$,
      (c) $N=10^6$ and (d) $N=10^7$, together with the solution of the corresponding deterministic model. Each curve shows the susceptible/ infectious/ recovered fraction ({\it i.e.} $S(t)/(S(t)+I(t)+R(t))$, $I(t)/(S(t)+I(t)+R(t))$, and $R(t)/(S(t)+I(t)+R(t))$) against time. Parameter values were chosen as follows: $\beta=20 {\rm\ year}^{-1}$, $\beta_1=0.4$, $\gamma=10 {\rm\ year}^{-1}$ and $\nu=1{\rm\ year}^{-1}$. Parameter values were chosen for illustrative purposes rather than to represent a specific disease in a specific population. In each case, the initial conditions of the system were taken to be $S(0)=0.92N$, $I(0)=0.08N$, with $R(0)=0$.
}
\end{figure}

In order to verify the diffusion-like behavior of $W_N(t)$, 4000
realizations of the model were generated. Figure 2(a) shows the
distribution of the values of the $j$th component of $W_N(1)$
calculated with an initial population size of $N=10^6$. As expected, this
marginal distribution is consistent with a normal
distribution. Furthermore, the mean and standard deviation of this
distribution agree with those predicted by the theory. Figure 2(b)
shows the distribution of the number of infectives population at time $t=1$
seen across the same set of realizations. 

\begin{figure}[H]
    \centering
    \subfigure[] {
    \includegraphics[width=0.475\columnwidth]{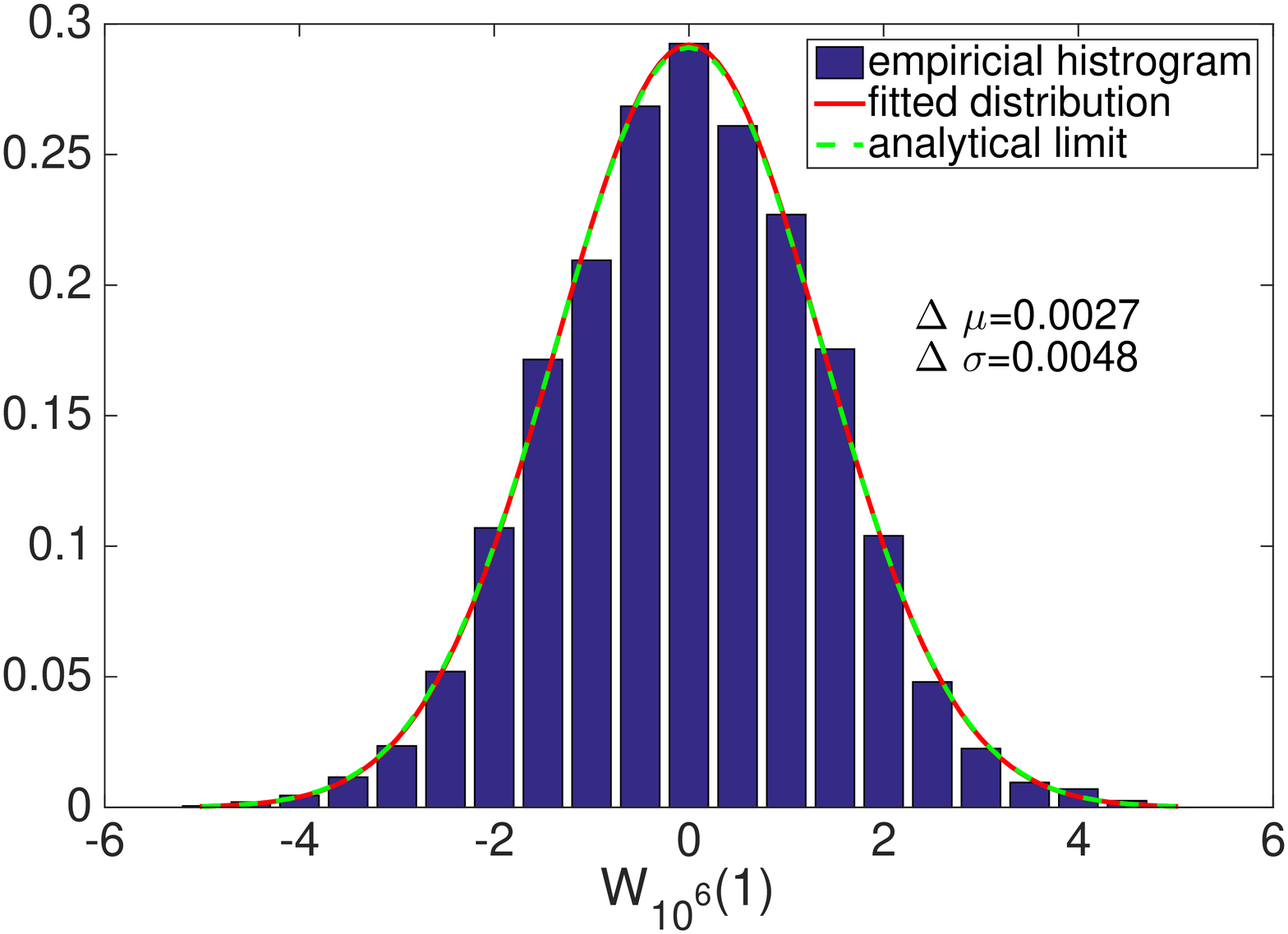} 
    } 
    \subfigure[] {
    \includegraphics[width=0.475\columnwidth]{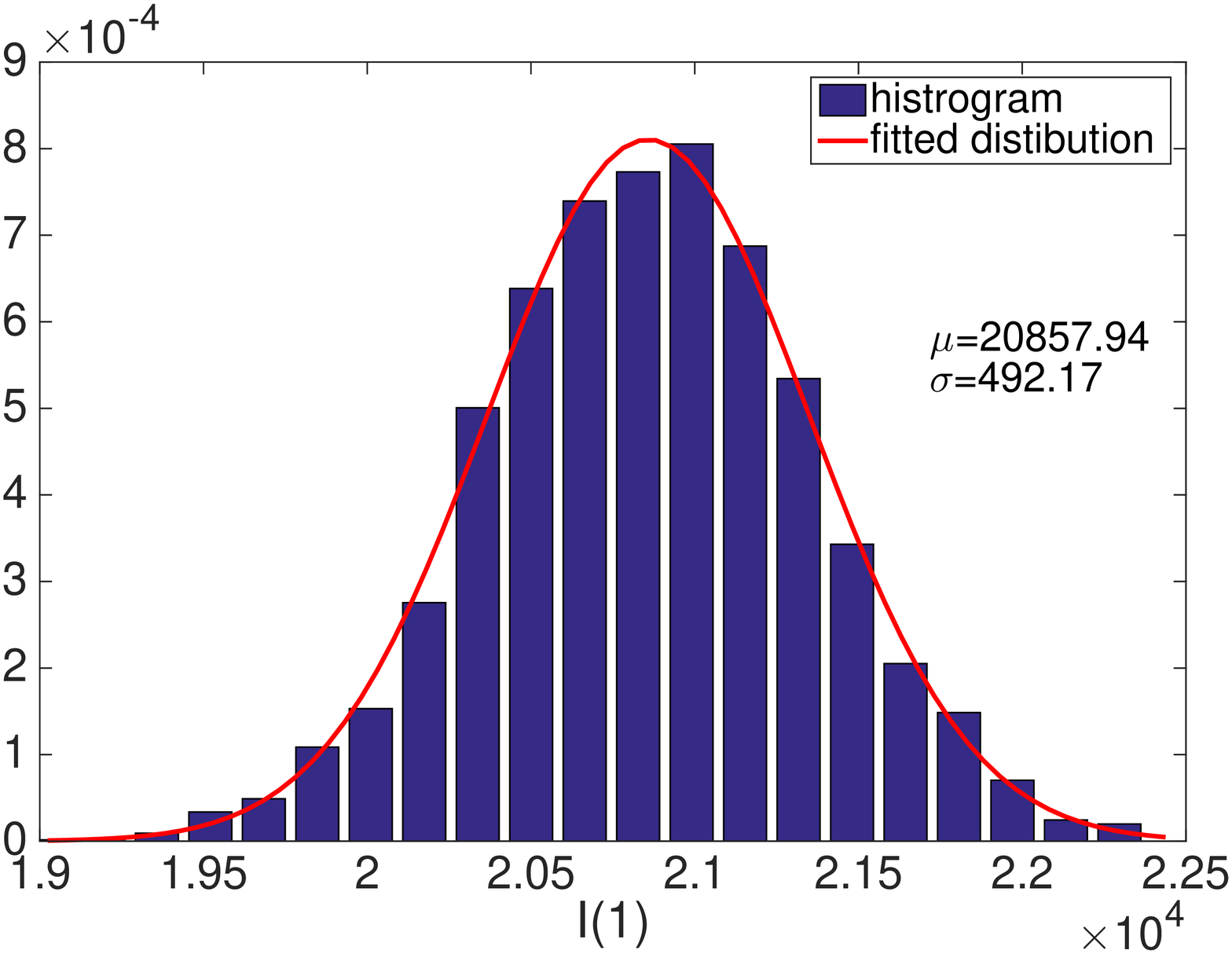} 
    } 
\caption{(a) Distribution of the $2$nd component (the component corresponding to the infective) of $W_{10^6}(1)$ calculated across 4000 realizations of the model (histogram). The solid red line indicates the best-fitting normal distribution fitted to the empirical histogram. The green dashed curve indicates the normal distribution predicted by the theory. (b) Distribution of the number of infectives seen at time $t=1$ across the same 4000 model realizations. Parameter values and initial conditions were chosen as in Figure 1. 
}
\end{figure}

Finally, Figure 3 explores the scaling of the standard deviation of
the difference between the realizations and their drift part with the
initial population size $N$. For each value of $N$, 4000 realizations of
the model were generated and, $\sigma^i_N(1)$, the standard deviation of the 2nd component of 
$$
Z_N(1)=\xi^N_1-\xi^N_0-\int_0^1 F(\xi^N_s,s)ds,
$$
was calculated across the set of realizations. We also calculate $f^i_N(1)$, the mean of fraction of the infectives across the set of realizations. Let 
$$
R^i_N(1)=\sigma^i_N(1)/ f^i_N(1)=O(N^{-1/2})
$$ 
be the ratios between them. The following figure shows, on a log-log scale, this value as a function of $N$. These points fall around a straight line of slope -1/2, consistent with the $1/\sqrt{N}$ scaling predicted by the theory.
\begin{figure}[H]
    \centering
    \includegraphics[width=4.5in]{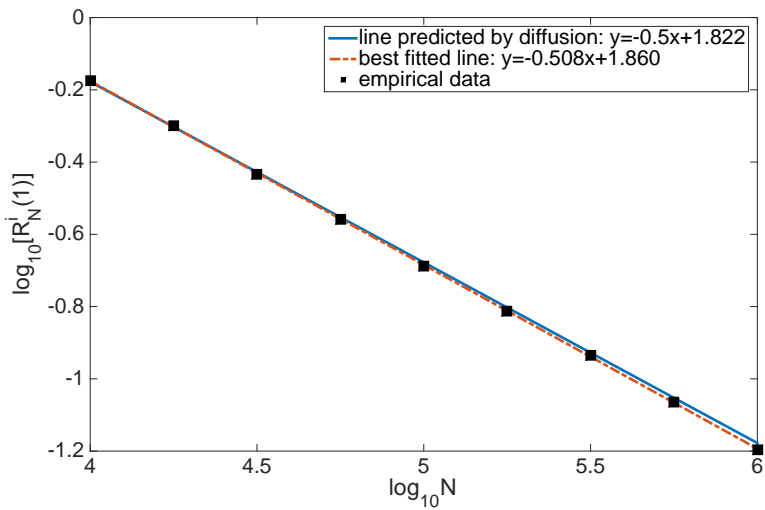}
 \caption{The data points give log-log plot of $R^i_N(1)$, calculated
   from 4000 realizations for each $N$ against the population size
   $N=10^4, 10^{4.25}, 10^{4.5},\cdots, 10^{6}$. The red dashed line
   gives the best fitted straight using those data points. The blue
   solid line is the line predicted by Theorem 1 and 2. Parameter values and initial conditions were chosen as in Figure 1. 
}
\end{figure}

\section*{Acknowledgements}

This work was carried out as part of the Mathematical and Statistical
Ecology program of the Statistical and Applied
Mathematical Sciences Institute (SAMSI), funded under NSF grant DMS-0635449. 
ALL is also supported by grants from the National Institutes of Health
(P01-AI098670) and the National Science Foundation (RTG/DMS-1246991). The authors wish to thank Rick Durrett and Jonathan Mattingly for fruitful
discussions.

\end{document}